\documentclass[11pt,a4paper]{amsart}

\usepackage[T1]{fontenc}

\usepackage{mathtools}
\usepackage{amssymb}
\usepackage{amsthm}
\usepackage{mathrsfs}
\usepackage{tikz-cd}
\usepackage{enumitem}
\usepackage{stackrel}
\usepackage{csquotes}

\usepackage[geometry]{ifsym}
\usepackage[margin=2cm]{geometry}

\usepackage[english]{babel}

\usetikzlibrary{babel}

\newcommand{\cat}[1]{\mathcal{#1}}
\newcommand{\bicat}[1]{\mathbb{#1}}
\newcommand{\catname}[1]{\mathsf{#1}}

\newcommand{\adj}{\dashv}
\newcommand{\dash}{\text{-}}
\newcommand{\iso}{\cong}
\newcommand{\eqv}{\simeq}
\renewcommand{\leq}{\leqslant}
\newcommand{\pt}{\boldsymbol{\cdot}}

\renewcommand{\Join}{\bigvee}
\newcommand{\meet}{\land}

\newcommand{\conn}{\catname{conn}}

\newcommand{\op}{\catname{op}}

\newcommand{\Cat}{\catname{Cat}}
\newcommand{\Nat}{\catname{Nat}}

\newcommand{\Top}{\catname{Top}}
\newcommand{\Stn}{\catname{Stn}}
\newcommand{\Set}{\catname{Set}}

\newcommand{\CHaus}{\catname{CHaus}}

\newcommand{\CAT}{\catname{CAT}}
\newcommand{\MndCat}{\catname{MndCat}}
\newcommand{\SymMndCat}{\catname{SymMndCat}}
\newcommand{\CartCat}{\catname{CartCat}}

\renewcommand{\hom}{\catname{hom}}

\newcommand{\CatB}{\Cat(\cat B)}
\newcommand{\CatV}{\Cat(\cat V)}

\newcommand{\VCat}{\cat V\dash \Cat}
\newcommand{\WCat}{\cat W\dash \Cat}
\newcommand{\FamV}{\Fam(\cat V)}
\newcommand{\FamVCat}{\FamV\dash\Cat}
\newcommand{\SetCat}{\Set\dash\Cat}

\newcommand{\TVCat}{(T,\cat V)\dash \Cat}
\newcommand{\CatTV}{\Cat(T, \cat V)}

\DeclareMathOperator{\id}{\catname{id}}
\DeclareMathOperator{\ob}{\catname{ob}}
\DeclareMathOperator{\Fam}{\catname{Fam}}
\DeclareMathOperator{\Desc}{\catname{Desc}}
\DeclareMathOperator{\PsPb}{\catname{PsPb}}
\DeclareMathOperator{\colim}{\catname{colim}}
\DeclareMathOperator{\Ker}{\catname{ker}}

\DeclareMathOperator{\mult}{\mathsf{c}}
\DeclareMathOperator{\unit}{\mathsf{u}}
\DeclareMathOperator{\multcomp}{\mathsf{m}}
\DeclareMathOperator{\unitcomp}{\mathsf{e}}
\newcommand{\cc}[1]{\mult_{#1}}
\newcommand{\uu}[1]{\unit_{#1}}
\newcommand{\m}[1]{\multcomp^{#1}}
\newcommand{\e}[1]{\unitcomp^{#1}}

\newtheorem{lemma}{Lemma}[section]
\newtheorem{proposition}[lemma]{Proposition}
\newtheorem{corollary}[lemma]{Corollary}
\newtheorem{theorem}[lemma]{Theorem}

\theoremstyle{definition}

\newtheorem{remark}[lemma]{Remark}

\usepackage[backend=biber,style=trad-plain]{biblatex}
\DefineBibliographyStrings{english}{pages={pp.}}

\begin{document}

  \title{On effective descent $\cat V$-functors and familial descent morphisms}

\author[R. Prezado]{Rui Prezado}

\address[1]{University of Coimbra, CMUC, Department of Mathematics,
Portugal}

\email[1]{rui.prezado@student.uc.pt}

\thanks{This research was partially supported  by the Centre for Mathematics
of the University of Coimbra - UIDB/00324/2020, funded by the Portuguese
Government through FCT/MCTES, and by the grant PD/BD/150461/2019 funded by
Fundação para a Ciência e Tecnologia (FCT)}

\keywords{effective descent morphisms, enriched category, pseudopullback,
Grothendieck descent theory}

\subjclass{18N10,18F20,18D20,18B50,18A25,18A35}

\date{May 06, 2023}

  \begin{abstract}
    We study effective descent \( \cat V \)-functors for cartesian monoidal
categories \( \cat V \) with finite limits. This study is carried out
via the properties enjoyed by the 2-functor \( \cat V \mapsto \Fam(\cat V) \),
results about effective descent of bilimits of categories, and the fact that
the enrichment 2-functor preserves certain bilimits. Since these results rely
on an understanding of (effective) descent morphisms in \( \FamV \), we
carefully study these morphisms in free coproduct completions. Finally, we
provide refined conditions when \( \cat V \) is a regular category.

  \end{abstract}

  \maketitle

  \setcounter{tocdepth}{1}
  \tableofcontents

  \section*{Introduction}
    \label{sect:introduction}
    Let \( \cat C \) be a category with pullbacks. For each morphism \( p \colon
x \to y \), we have a change-of-base functor along \(p\):
\begin{equation*}
  p^* \colon \cat C/y \to \cat C/x
\end{equation*}
Via these functors, we are able to provide a description of the \textit{basic
bifibration} of \( \cat C \). Thanks to the Bénabou-Roubaud theorem
\cite{BR70} (see also \cite[p. 258]{JT94} or \cite[Theorems 7.4 and
8.5]{Luc18a} for generalisations), the \textit{descent category} for \(p\)
with respect to the basic bifibration, denoted \( \Desc(p) \), is equivalent
to the Eilenberg-Moore category for the monad induced by the adjunction \( p_!
\adj p^* \). This allows us to say that the morphism \(p\) is
\textit{effective for descent} if the comparison functor \( \mathcal K^p \) in
the Eilenberg-Moore factorisation~\eqref{eq:em.fact}
\begin{equation}
  \label{eq:em.fact}
  \begin{tikzcd}
    \cat C/y \ar[rd,"p^*",swap] \ar[rr,"\mathcal K^p"] 
      && \Desc(p) \ar[ld,"\mathcal U^p"] \\
    & \cat C/x 
  \end{tikzcd}
\end{equation}
is an equivalence of categories; here, \( \mathcal U^p \) is the functor which
forgets descent data.

Janelidze-Galois theory \cite{BJ01} and Grothendieck descent theory
\cite{JT97, Luc18a} feature the use of effective descent morphisms, requiring
the knowledge of some (or all) such morphisms in the category of interest, and
are the main motivation to undertake the study of finding sufficient
conditions or even characterising effective descent morphisms; see
\cite{JST04, JT94} for introductions to the subject.

As an example, if \( \cat C \) is a locally cartesian closed category, or an
exact category (in the sense of Barr \cite{Bar71}), the effective descent
morphisms are precisely the regular epimorphisms. However, the
characterisation of effective descent morphisms in a given category \( \cat C
\) is a notoriously difficult problem in general; for instance, see the
characterisation in \cite{RT94} and a subsequent reformulation \cite{CH02} for
the case \( \cat C = \Top \).  

Motivated by this reformulation, \cite{CH04, CH12, CH17, CJ11} study this
characterisation problem for more general notions of \textit{spaces}: these
works provide various results about effective descent in \((T, \cat
V)\)\textit{-categories} (originally defined in \cite{CT03}). Due to their
concerns with topological results, the study was restricted to the case in
which the enriching category \( \cat V \) is a quantale.

From the perspective of internal structures, we have the work of Le Creurer
\cite{Cre99}, in which he studies the problem of effective descent morphisms
for essentially algebraic structures internal to a category \( \cat B \) with
finite limits. In particular, the author provides sufficient conditions for
effective descent morphisms in \( \cat C = \CatB \), and confirms that these
conditions provide a complete characterisation with the added requirement that
\( \cat B \) is extensive and has a (regular epi, mono)-factorization system.
Generalisations of these results to internal \textit{multicategories} were
studied in \cite{PL23}.

Let \( - \pt 1 \colon \Set \to \cat V \) be the ``copower with the terminal
object'' functor, defined on objects by \( X \mapsto X \pt 1 = \sum_{x \in X}
1 \). Making use of Le Creurer's results, Lucatelli Nunes, via his study on
effective descent morphisms for bilimits of categories, provides sufficient
conditions for effective descent morphisms in \( \cat C = \VCat \) via the
following pseudopullback (see \cite[Lemma 9.10, Theorem 9.11]{Luc18a}):
\begin{equation}
  \label{eq:lucatelli.sq}
  \begin{tikzcd}
    \VCat \ar[r] \ar[d,"\ob",swap] & \CatV \ar[d,"(-)_0"] \\
    \Set \ar[r,"-\pt 1",swap] & \cat V
  \end{tikzcd}
\end{equation}
for suitable lextensive categories \( \cat V \), with the \textit{cartesian}
monoidal structure.

The central contribution of this paper is to extend \cite[Theorem
9.11]{Luc18a} to all cartesian monoidal categories \( \cat V \) with finite
limits, in Theorem~\ref{thm:main.result}. We highlight the use of the
following three tools, used in the proof of Lemma~\ref{lem:main.result}, which
are the skeleton of the argument: the properties of familial 2-functors, in
particular, of the endo-2-functor \( \Fam \colon \CAT \to \CAT \) studied in
\cite{Web07}; results about effective descent morphisms in bilimits of
categories (see \cite[Theorem 9.2 and Corollary 9.5]{Luc18a}); and
preservation of pseudopullbacks via enrichment
(Theorem~\ref{thm:enrichment.preserves.pspb}).

Since Theorem~\ref{thm:main.result} relies on understanding (effective)
descent morphisms in \( \FamV \), it naturally raises the problem of studying
these classes of epimorphisms in the free coproduct completion of \( \cat V
\).  Lemma~\ref{lem:fam.reg.epis} and Theorem~\ref{cor:fam.eff.desc} provide a
couple of improvements, which we first illustrate in
Theorem~\ref{thm:chey.cat.ef.desc} for \( \cat V \) a (co)complete Heyting
lattice (a new proof for one implication of \cite[Theorem 5.4]{CH04}), and
then we apply it to obtain the more general Theorem \ref{thm:reg.cat.desc},
providing a refinement of Theorem~\ref{thm:main.result} for regular categories
\( \cat V \).

In Section~\ref{sect:preliminaries}, we recall the notion of pseudopullback
in the restricted context of the 2-categories \( \Cat \) and \( \MndCat \), we
fix some terminology and notation for (strong) monoidal functors, used in the
proofs of the results in Section~\ref{sect:helpings}, and we recall a couple
of results from \cite{Luc18a} and \cite{Web07}, restated in a convenient form,
which are part of our toolkit in Section~\ref{sect:main}.

Section~\ref{sect:helpings} is devoted to establishing some technical results
on preservation of pseudopullbacks
(Theorem~\ref{thm:enrichment.preserves.pspb}), full faithfulness
(Lemma~\ref{lem:enrichment.preserves.ff}) by the 2-functor \( (-)\dash \Cat
\colon \MndCat \to \CAT \), and in the preservation of descent morphisms by
suitable functors (Lemma~\ref{lem:ff.left.right.adjs}), which complete our
toolkit.

As mentioned earlier, we establish our main result in Section~\ref{sect:main},
in Theorem \ref{thm:main.result}. We restate the main ideas here: if \( \cat V
\) is a cartesian monoidal category, the following composite of functors
\begin{equation}
  \label{eq:intro.refl}
  \begin{tikzcd}
    \VCat \ar[r] & \FamVCat \ar[r] & \Cat(\FamV)
  \end{tikzcd}
\end{equation}
reflects effective descent morphisms. If \( F \) is a \( \cat V \)-functor, we
let \( \tilde F \) be the value of \eqref{eq:intro.refl} at \(F\). Then, if we
denote by \( \tilde F^n \) the underlying morphism in \( \FamV \) on the
objects of \(n\)-tuples of composable morphisms, we have that if
\begin{itemize}[label=--]
  \item
    \( \tilde F^1 \) is an effective descent morphism in \( \FamV \),
  \item
    \( \tilde F^2 \) is a descent morphism in \( \FamV \),
  \item
    \( \tilde F^3 \) is an almost descent morphism in \( \FamV \),
\end{itemize}
then \( \tilde F \) is an effective descent morphism in \( \Cat(\FamV) \) by
\cite[Proposition 3.3]{Cre99}, and by reflection, \( F \) is an effective
descent morphism in \( \VCat \).

Indeed, these conditions on \(F\) are statements about (effective) (almost)
descent morphisms in \( \FamV \), leading us to studying such morphisms in the
coproduct completion of \( \cat V \). We devote Section~\ref{sect:examples} to
provide tractable descriptions of these classes of epimorphisms, with an
illustrative application to categories enriched in (co)complete Heyting
lattices. We obtain Theorem~\ref{thm:reg.cat.desc}, which refines
Theorem~\ref{thm:main.result} for regular categories, with further
simplifications for infinitary coherent categories, exact categories or
locally cartesian closed categories. 

In Section \ref{sect:examples.pt2}, we provide some brief remarks regarding 
categories enriched in cartesian monoidal categories. We instantiate our
results when the enriching category \( \cat V \) is
\begin{itemize}[label=--]
  \item
    the category \( \CHaus \) of compact Hausdorff spaces,
  \item
    the category \( \Stn \) of Stone spaces,
  \item
    a (Grothendieck) topos.
\end{itemize}

Finally, we have a couple of concluding remarks in
Section~\ref{sect:future.work}, where we sketch some possible lines of future
research, with regard to extending the result to all symmetric monoidal
categories, or to generalized multicategories. 

\subsection*{Acknowledgments}

The author is deeply grateful to Fernando Lucatelli Nunes and Maria Manuel
Clementino for their helpful comments regarding this work. The valuable
comments of the anonymous referee were also very appreciated.

  \section{Preliminaries}
    \label{sect:preliminaries}
    The purpose of this section is to give a concise summary of the terminology
and notation used for our main results. We begin by reviewing the notion of
\textit{pseudopullbacks}, as our treatment of effective descent morphisms
employs the techniques of \cite{Luc18a} on commutativity of bilimits.
Then, we recall the notion of \textit{free coproduct completion} of a
category, an important tool for our main insight. Finally, we fix some
notation regarding \textit{monoidal} and \textit{enriched} categories.

\subsection{Pseudopullbacks and 2-pullbacks:}

Let \( F \colon \cat C \to \cat E \) and \( G \colon \cat D \to \cat E \) be
functors. The \textit{pseudopullback} of \(F,G\), denoted by \( \PsPb(F,G) \)
may be succinctly defined as \textit{the full subcategory} of the comma
category \( (F \downarrow G) \) whose objects are \textit{isomorphisms}. To be
explicit, \( \PsPb(F,G) \) has
\begin{itemize}[label=--]
  \item
    objects given by isomorphisms \( \xi \colon Fc \iso Gd \), where \( c
    \in \cat C \) and \( d \in \cat D \),
  \item
    morphisms \( (\zeta \colon Fa \to Gb) \to (\xi \colon Fc \to Gd) \) given
    by a pair of morphisms \(f \colon a \to c \) and \( g \colon b \to d \)
    such that \( \xi \circ Ff = Gg \circ \zeta \).
  \item
    identities and composition given componentwise from \( \cat C \) and \(
    \cat D \).
\end{itemize}
The \textit{2-pullback} of \( F, G \) is simply the ordinary pullback in the
underlying category \( \CAT \). It may also be seen as the full subcategory of
\( \PsPb(F,G) \) whose objects are the identity morphisms; these are
determined by pairs \( c \in \cat C \), \( d \in \cat D \) such that \( Fc =
Gd \).

We remark that 2-pullbacks and pseudopullbacks are far from being equivalent
in general. We consider the following cospan:
\begin{equation}
  \label{eq:cospn}
  \begin{tikzcd}
    1 \ar[r,"x"] & (x \iso y) & 1 \ar[l,"y",swap]
  \end{tikzcd}
\end{equation}
where \( 1 \) is the terminal category,  and \( (x \iso y) \) is a category
with two isomorphic, but distinct, objects \(x, y\), identified by the
functors \( x,\,y  \colon 1 \to (x \iso y) \). The 2-pullback is the empty
category, since \( x \neq y \), while the pseudopullback is the discrete
category whose objects are the isomorphisms from \( x \) to \( y \) in \( (x
\iso y) \).

Regarding descent theory, we recall the following result of Lucatelli Nunes:
\begin{proposition}[{\cite[Corollary 9.6]{Luc18a}}]
  \label{prop:pspb.descent}
  Suppose Diagram~\eqref{eq:pb.cat.sq} below is a pseudopullback of
  categories with pullbacks and pullback-preserving functors
  \begin{equation}
    \label{eq:pb.cat.sq}
    \begin{tikzcd}
      \cat A \ar[r,"F"] \ar[d,"H",swap] 
        & \cat B \ar[d,"K"] \\
      \cat C \ar[r,"G",swap] & \cat D
    \end{tikzcd}
  \end{equation}
  Let \( f \) be a morphism in \( \cat A \). If
  \begin{itemize}[label=--]
    \item
      \( Ff \) is effective for descent,
    \item
      \( Hf \) is effective for descent,
    \item
      \( KFf \iso GHf \) is a descent morphism,
  \end{itemize}
  then \(f\) is effective for descent.
\end{proposition}

\subsection{Free coproduct completion of a category:}
\label{subsect:fam}

Let \( \cat V \) be a category. The \textit{free coproduct completion} of \(
\cat V \), also known as the category of \textit{families of objects in} \(
\cat V \), is denoted by \( \FamV \). Its objects are set-indexed families 
\begin{equation*}
  (X_j)_{j \in J}
\end{equation*}
of objects \( X_j \) in \( \cat V \), and a morphism 
\begin{equation}
  \label{eq:fam.morph}
  (X_j)_{j \in J} \to (Y_k)_{k \in K} 
\end{equation}
consists of a pair \( (f,\phi) \) where \( f \colon J \to K \) is a function
on the underlying sets, and 
\begin{equation}
  \label{eq:fam.fam.morph}
  \phi = (\phi_j \colon X_j \to Y_{fj})_{j \in J} 
\end{equation}
is a set-indexed family of morphisms \( \phi_j \) in \( \cat V \). We refer to
\cite{Bén85, CLW93, BJ01} for more thorough introductions to this concept.
Here, we simply recall that the ``underlying set'' functor \( \FamV \to \Set
\) is a \textit{fibration}, and that we have a canonical, fully faithful
functor \( \eta_{\cat V} \colon \cat V \to \FamV \) identifying the
one-element families. More importantly, we recall the following observation of
Weber:

\begin{proposition}[{\cite[Proposition 5.15]{Web07}}]
  \label{prop:fam.unit.cartesian}
  The canonical, fully faithful functors \( \eta_{\cat V} \colon \cat V \to
  \Fam(\cat V) \) form a 2-natural cartesian transformation, that is, 
  for all functors \(F\colon \cat V \to \cat W\), Diagram~\eqref{eq:2pb.eta}
  below is a 2-pullback.
  \begin{equation}
    \label{eq:2pb.eta}
    \begin{tikzcd}
      \cat V \ar[d,"F",swap] \ar[r,"\eta_{\cat V}"]
        & \Fam(\cat V) \ar[d,"\Fam(F)"] \\
      \cat W \ar[r,"\eta_{\cat W}",swap]
        & \Fam(\cat W)
    \end{tikzcd}
  \end{equation}
\end{proposition}

A final remark on notation: for a morphism \eqref{eq:fam.morph} in \( \FamV \)
with \( K \iso 1 \), the function on the underlying sets is uniquely
determined. In this case, the underlying family \eqref{eq:fam.fam.morph} of
morphisms in \( \cat V \) is sufficient to determine the morphism in \( \FamV
\). For this reason, we simply denote such morphisms as \( \phi \colon
(X_j)_{j \in J} \to Y \).

\subsection{Monoidal categories:}

Let \( \cat V = (\cat V, \otimes, I) \) and \( \cat W = (\cat W, \otimes, I)
\) be monoidal categories, whose coherence isomorphisms we omit, as they play
no role in what follows. We recall that a \textit{monoidal functor} \( F
\colon \cat V \to \cat W \) consists of a functor \( F \) between the
underlying categories, preserving the unit object and tensor products only
up-to-isomorphism. This means that we have 
\begin{itemize}[label=--]
  \item
    an isomorphism \( \e F \colon I \to FI \), the \textit{unit comparison}
    morphism,
  \item
    and an isomorphism \( \m F \colon Fx \otimes Fy \to F(x \otimes y) \), for
    each pair \(x, y \) of objects, the \textit{tensor comparison} morphism
\end{itemize}
satisfying naturality and coherence conditions (see \cite[p. 1889]{Bén63}). 

We denote by \( \MndCat \) the 2-category of monoidal categories, monoidal
functors, and their natural transformations \cite[p. 474]{EK66}.  We further
highlight that the forgetful 2-functor \( \MndCat \to \Cat \) is pseudomonadic
\cite[Section 3.1]{Lei04}, \cite[Remark 4.3]{Luc18b}, and therefore it creates
bilimits. In particular, the underlying category of the pseudopullback of a
cospan of monoidal functors is the pseudopullback of the underlying ordinary
functors, and fully faithful morphisms in \( \MndCat \) are precisely those
monoidal functors whose underlying functor is fully faithful in \( \Cat \). 

We have a 2-functor \( (-)\dash \Cat \colon \MndCat \to \CAT \), mapping each
monoidal category \( \cat V \) to the category \( \VCat \) of small \( \cat V
\)-categories, and for each monoidal functor \( F \colon \cat V \to \cat W \), 
we have the direct image functor \( F_! \colon \VCat \to \WCat \). On a small
\( \cat V \)-category \( \cat C \), the \( \cat W \)-category \( F_!\cat C \)
has the same underlying set of objects, and for each pair of objects \(x,y\)
in \( \cat C \), \( (F_!\cat C)(x,y) = F\cat C(x,y) \). 

Regarding notation, we will denote the \textit{unit and composition morphisms}
of a \( \cat V \)-category \( \cat C \) by \( \uu{\cat C} \colon I \to \cat
C(x,x) \) for each object \( x \) in \( \cat C \), and \( \cc{\cat C} \colon
\cat C(y,z) \otimes \cat C(x,y) \to \cat C(x,z) \) for each triple \( x,y,z \)
of objects in \( \cat C \). The unit and composition morphisms for \( F_!\cat
C \) are given by
\begin{equation*}
  \uu{F_!\cat C} = F\uu{\cat C} \circ \e F
  \qquad \text{and} \qquad
  \cc{F_!\cat C} = F\cc{\cat C} \circ \m F.
\end{equation*}

As it will prove to be convenient, we fix the following notation of composable
pairs and triples of hom-objects of a \( \cat V \)-category \( \cat C \): 
\begin{itemize}[label=--]
  \item
    \( \cat C(x,y,z) = \cat C(y,z) \otimes \cat C(x,y) \), 
  \item
    \( \cat C(w,x,y,z) = \cat C(x,y,z) \otimes \cat C(w,x) \).
\end{itemize}
For example, the composition morphism of \( \cat C \) may be written as \(
\cc{\cat C} \colon \cat C(x,y,z) \to \cat C(x,z) \). We may also write \( \m F
\colon (F_!\cat C)(x,y,z) \to F(\cat C(x,y,z)) \) for the tensor comparison
isomorphism when hom-objects are concerned.  Analogously, we may define 
\begin{equation}
  \Phi_{x,y,z} = \Phi_{y,z} \otimes \Phi_{x,y} 
    \colon \cat C(x,y,z) \to \cat D(\Phi x, \Phi y, \Phi z) 
\end{equation}
for a \( \cat V \)-functor \( \Phi \colon \cat C \to \cat D \).

  \section{Preservation of bilimits and descent}
    \label{sect:helpings}
    The following result, present in a more general form in \cite{FL22}, is a
helpful, labour-saving device in our work with pseudopullbacks:

\begin{theorem}
  \label{thm:enrichment.preserves.pspb}
  The enrichment 2-functor \( (-)\dash \Cat \colon \MndCat \to \CAT \) 
  preserves pseudopullbacks.
\end{theorem}

\begin{proof}
  Let \( F \colon \cat U \to \cat W \) and \( G \colon \cat V \to \cat W \) be
  monoidal functors between monoidal categories.

  We desire to confirm that \( \PsPb(F, G)\dash \Cat \eqv \PsPb(F_!, G_!) \);
  let \( \Phi \colon F_!\cat B \iso G_!\cat C \) be an isomorphism of \( \cat W
  \)-categories, where \( \cat B \) is a \( \cat U \)-category and \( \cat C
  \) is a \( \cat V \)-category. We define a \( \PsPb(F,G) \)-category \( \cat
  D_\Phi \) with
  \begin{itemize}[label=--]
    \item
      set of objects given by \( \ob \cat D_\Phi = \ob \cat B \),
    \item
      hom-object given by \( \cat D_\Phi(x,y) = \Phi_{x,y} \colon F\cat B(x,y)
      \iso G\cat C(\Phi x,\Phi y) \) at \(x,y \in \ob \cat D_\Phi \),
    \item
      unit object and composition given by the pairs \( (\uu{\cat B}, \uu{\cat
      C}) \), \( (\cc{\cat B}, \cc{\cat C}) \) of the respective unit objects
      and compositions from \( \cat B \) and \( \cat C \); these pairs are
      well-defined morphisms of \( \PsPb(F,G) \), since \(F,G\) are monoidal
      functors and \( \Phi \) is a \( \cat W \)-functor.
  \end{itemize}
  To be more precise on this last point, we note that the following diagrams
  commute:
  \begin{equation}
    \begin{tikzcd}
      && I_{\cat W} \ar[dll,"\e F",swap] \ar[drr,"\e G"] 
                    \ar[ddl,"\uu{F_!\cat B}",swap] 
                    \ar[ddr,"\uu{G_!\cat C}"] \\
      FI_{\cat U} \ar[rd,"F\uu{\cat B}",swap]
        &&&& GI_{\cat V} \ar[ld,"G\uu{\cat C}"] \\
      & F\cat B(x,x) \ar[rr,"\Phi_{x,x}",swap]
        && G\cat C(\Phi x,\Phi x)
    \end{tikzcd}
  \end{equation}
  \begin{equation}
    \begin{tikzcd}[column sep=large]
      F(\cat B(x,y,z)) \ar[d,"F\cc{\cat B}",swap]
      & (F_!\cat B)(x,y,z) 
        \ar[l,"\m F",swap] \ar[r,"\Phi_{x,y,z}"] 
        \ar[ld,"\cc{F_!\cat B}" description]
      & (G_!\cat C)(\Phi x,\Phi y,\Phi z) 
        \ar[d,"\m G"] \ar[ld,"\cc{G_!\cat C}" description] \\
      F\cat B(x,z) \ar[r,"\Phi_{x,z}",swap]
      & G\cat C(\Phi x, \Phi z)
      & G(\cat C(\Phi x, \Phi y,\Phi z))
          \ar[l,"G\cc{\cat C}"] 
    \end{tikzcd}
  \end{equation}
  Since identity and associativity laws of \( \cat D_\Phi \) are precisely
  those of \( \cat B \) and \( \cat C \), it follows that \( \cat D_\Phi \) is
  indeed well-defined. 

  The underlying \( \cat U \)-category of \( \cat D_\Phi \) is \( \cat B \)
  itself, while its underlying \( \cat V \)-category is isomorphic to \( \cat
  C \): it is given by \( \ob \Phi \) on the sets of objects, and identity on
  the hom-objects.

  Moreover, let \( \cat X \), \( \cat Y \) be \(\PsPb(F,G)\)-categories, and
  let \( H \colon \cat X_{\cat U} \to \cat Y_{\cat U} \) be a \( \cat U
  \)-functor and \( K \colon \cat X_{\cat V} \to \cat Y_{\cat V} \) be a \(
  \cat V \)-functor between the underlying \( \cat U \)-categories and \( \cat
  V \)-categories of \( \cat X \) and \( \cat Y \) respectively, such that
  \( \ob H = \ob K \) and
  \begin{equation}
    \label{eq:why.why.why}
    GK_{x,y} \circ \cat X(x,y) = \cat Y(Hx,Hy) \circ FH_{x,y}. 
  \end{equation}
  We note that there exists a unique \( \PsPb(F,G) \)-functor \( \Phi \colon
  \cat X \to \cat Y \) with underlying \( \cat U \)-functor and \( \cat V
  \)-functor given by \(H\) and \(K\), respectively. Indeed, \( \Phi \colon
  \cat X \to \cat Y \) is given as follows:
  \begin{itemize}[label=--]
    \item
      \( \ob \Phi = \ob H \),
    \item
      \( \Phi_{x,y} \) is given by the pair \( H_{x,y} \), \( K_{x,y} \),
      which is a morphism 
      \begin{equation*}
        (\cat X(x,y) \colon F\cat X_{\cat U}(x,y) 
                              \iso G\cat X_{\cat V}(x,y)) 
          \to (\cat Y(Hx,Hy) \colon F\cat Y_{\cat U}(Hx,Hy) 
                                      \iso G\cat Y_{\cat V}(Hx,Hy))
      \end{equation*}
      in \( \PsPb(F,G) \), due to \eqref{eq:why.why.why}.
  \end{itemize}
  The laws that make \( \Phi \) into a \( \PsPb(F,G) \)-functor are precisely
  given by the laws that make \(H\) into a \( \cat U \)-functor and \( K \)
  into a \( \cat V \)-functor.

  If \( \Psi \colon \cat X \to \cat Y \) is a \( \PsPb(F,G) \)-functor with
  \(H\) as underlying \( \cat U \)-functor and \( K \) as underlying \( \cat V
  \)-functor, we necessarily get \( \Phi = \Psi \) by comparing their
  hom-morphisms.
\end{proof}

\begin{lemma}
  \label{lem:enrichment.preserves.ff}
  The enrichment 2-functor \( (-)\dash \Cat \colon \MndCat \to \CAT \)
  preserves fully faithful functors.
\end{lemma}

\begin{proof}
  Let \( F \colon \cat V \to \cat W \) be a fully faithful monoidal functor.
  To prove \( F_! \colon \VCat \to \WCat \) is fully faithful, let \( \cat C,
  \cat D \) be \( \cat V \)-categories, and let \( \Psi \colon F_!\cat C \to
  F_!\cat D \) be a \( \cat W \)-functor. It consists of the following
  data:
  \begin{itemize}[label=--]
    \item
      A function \( \Psi \colon \ob \cat C \to \ob \cat D \),
    \item
      A morphism \( \Psi_{x,y} \colon F\cat C(x,y) \to F\cat D(\Psi x, \Psi y)
      \) for each pair \(x,y \in \ob \cat C \).
  \end{itemize}
  Since \(F\) is fully faithful, there exists a unique \( \Phi_{x.y} \colon
  \cat C(x,y) \to \cat D(\Psi x, \Psi y) \) such that \( F\Phi_{x,y} =
  \Psi_{x,y} \).

  With this, we define a \( \cat V \)-functor \( \Phi \colon \cat C \to \cat D
  \) given
  \begin{itemize}[label=--]
    \item
      on objects by the function \( \Phi = \Psi \colon \ob \cat C \to \ob \cat D \),
    \item
      on morphisms by \( \Phi_{x,y} \colon \cat C(x,y) \to \cat D(\Phi x, \Phi y)
      \) for each pair \( x, y \in \ob \cat C \).
  \end{itemize}
  This is a \( \cat V \)-functor: we note that the following diagrams commute
  \begin{equation}
    \begin{tikzcd}
      & I \ar[d,"\e F"] \ar[ddl,bend right,"u",swap] 
                       \ar[ddr,bend left,"u"] \\
      & FI \ar[ld,"Fu",swap] \ar[rd,"Fu"] \\
      FC(x,y) \ar[rr,"F\Phi_{x,y}",swap] && F\cat D(\Phi x,\Phi y)
    \end{tikzcd}
  \end{equation}
  \begin{equation}
    \begin{tikzcd}[column sep=large]
      (F_!\cat C)(x,y,z) \ar[r,"(F_!\Phi)_{x,y,z}"] 
                         \ar[d,"\m F",swap]
        & (F_!\cat D)(\Phi x, \Phi y, \Phi z) 
          \ar[d,"\m F"] \\
      F(\cat C(x,y,z)) \ar[r,"F\Phi_{x,y,z}"]
          \ar[d,"Fc",swap]
        & F(\cat D(\Phi x, \Phi y, \Phi z))
          \ar[d,"Fc"] \\
      F\cat C(x,z) \ar[r,"F \Phi_{x,z}",swap]
        & F\cat D(\Phi x, \Phi z)
    \end{tikzcd}
  \end{equation}
  so, by the full faithfulness of \(F\), plus invertibility of \( \e F \) and
  \( \m F \), we confirm that \( \Phi \) is a \( \cat V \)-functor. Moreover,
  by definition, it is the unique \( \cat V \)-functor such that \(F_!\Phi =
  \Psi \), which concludes our proof.
\end{proof}

\begin{lemma}
  \label{lem:ff.left.right.adjs}
  Given a string of adjoint functors \( L \adj F \adj R \) between categories
  with finite limits, if \( L \) (and therefore \( R \)) is fully faithful,
  then \(F\) preserves descent morphisms.
\end{lemma}

Before providing the proof, we recall that descent morphisms in categories
with finite limits are precisely the pullback-stable regular epimorphisms
(see, for instance, \cite{JST04, Cre99}).

\begin{proof}
  Let \(p \colon x \to y \) be a descent morphism. Since \(F\) is a left
  adjoint, we may conclude that \(Fp\) is a regular epimorphism; we just need
  to prove that it is stable under pullback.

  To do so, let \( f \colon z \to Fy \) be a morphism, and we consider the
  following pullback diagram:
  \begin{equation}
    \begin{tikzcd}
      f^*(Fx) \ar[r,"f^*(Fp)"] \ar[d] & z \ar[d,"f"] \\
      Fx \ar[r,"Fp",swap] & Fy
    \end{tikzcd}
  \end{equation}
  We wish to prove that \( f^*(Fp) \) is a regular epimorphism. Indeed, we
  note that \( FLf^*(Fp) \iso f^*(Fp) \), and since \(F\) reflects pullbacks
  (via \(R\)), we have a pullback
  \begin{equation}
    \begin{tikzcd}
      Lf^*(Fx) \ar[r,"Lf^*(Fp)"] \ar[d] & Lz \ar[d,"f^\sharp"] \\
      x \ar[r,"p",swap] & y
    \end{tikzcd}
  \end{equation}
  so that \( Lf^*(Fp) \iso (f^\sharp)^*(p) \) is a regular epimorphism; which
  is preserved by \(F\), hence \( f^*(Fp) \) must be a regular epimorphism, as
  desired.
\end{proof}

\begin{remark}
  \label{rem:two.examples}
  We highlight one application of Lemma~\ref{lem:ff.left.right.adjs}: for a
  category \( \cat B \) with finite limits, the underlying object-of-objects
  functor
  \begin{equation*}
    (-)_0 \colon \CatB \to \cat B
  \end{equation*}
  has fully faithful left and right adjoints: these assign to each object
  \( b \) of \( \cat B \) its respective \textit{discrete} and
  \textit{indiscrete} internal categories with \(b\) as the underlying object
  of objects; see \cite[\,\!{7.2.6}]{Jac99}. Thus, we conclude that \( (-)_0
  \) preserves descent morphisms. 
\end{remark}

Remark \ref{rem:two.examples} can be used to verify that \( \VCat \to \CatV \)
reflects effective descent morphisms for extensive categories \( \cat V \)
with finite limits with \( - \pt 1 \colon \Set \to \cat V \) fully faithful,
\textit{without} requiring \( \cat V \) to have a (regular epi,
mono)-factorization system, using the same argument as in the proof of
\cite[Theorem 9.11]{Luc18a}.

  \section{Descent for cartesian enriched categories}
    \label{sect:main}
    Throughout this section, fix a cartesian monoidal category \( \cat V \) with
finite limits, and consider the canonical embedding \( \eta \colon \cat V \to
\Fam(\cat V) \), as given in Subsection~\ref{subsect:fam}.

\begin{lemma}
  \label{lem:main.result}
  The direct image functor \( \eta_! \colon \VCat \to \FamVCat \) reflects
  effective descent morphisms. 
\end{lemma}

\begin{proof}
  By Proposition~\ref{prop:fam.unit.cartesian}, Diagram~\eqref{eq:key.square}
  \begin{equation}
    \label{eq:key.square}
    \begin{tikzcd}
      \cat V \ar[r,"\eta"] \ar[d]
        & \Fam(\cat V) \ar[d] \\
      1 \ar[r,"\eta",swap] & \Set
    \end{tikzcd}
  \end{equation}
  is a 2-pullback. Since \( \Fam(\cat V) \to \Set \) is an (iso)fibration, it
  follows that Diagram~\eqref{eq:key.square} is a pseudopullback, by
  \cite[Theorem 1]{JS93}. It is preserved by \( (-)\dash \Cat \), as shown in
  Theorem~\ref{thm:enrichment.preserves.pspb}, so we obtain the pseudopullback
  given in \eqref{eq:pspb} below:
  \begin{equation}
    \label{eq:pspb}
    \begin{tikzcd}
      \VCat \ar[r,"\eta_!"] \ar[d,"\ob",swap] 
        & \FamVCat \ar[d] \\
      \Set \ar[r,"\eta_!",swap] & \SetCat
    \end{tikzcd}
  \end{equation}

  To conclude the proof, we note that since \( (-)\dash \Cat \) is a
  2-functor, it preserves adjoints, which, together with
  Lemma~\ref{lem:enrichment.preserves.ff}, guarantees that the functor \(
  \FamVCat \to \SetCat \) has fully faithful left and right adjoints, thus it
  preserves descent morphisms by Lemma~\ref{lem:ff.left.right.adjs}. It is
  well-established that \( \eta_! \colon \Set \to \SetCat \) reflects descent
  morphisms, and descent morphisms in \( \Set \) are effective for descent.

  This places us under the conditions of Proposition~\ref{prop:pspb.descent},
  so the result follows.
\end{proof}

\begin{lemma}
  \label{lem:fam.thing}
  The category \( \FamV \) is extensive with finite limits, and \( - \pt 1
  \colon \Set \to \FamV \) is fully faithful.
\end{lemma}
\begin{proof}
  We have already confirmed that \( - \pt 1 \colon \Set \to \FamV \) is fully
  faithful in Remark~\ref{rem:two.examples}, because \( \cat V \) has a
  terminal object. Moreover, extensivity of \( \Fam(\cat V) \) is
  well-established; see, for instance, \cite[Proposition 2.4]{CLW93}.

  Existence of finite limits is a direct corollary of \cite[Proposition 4.1,
  Theorem 4.2]{Gra66}; we consider the fibration \( \Fam(\cat V) \to \Set \).
  The base category \(\Set\) has all (finite) limits, and the fibers \( \cat
  V^J \) at a set have finite limits as well. The latter are preserved by
  the change-of-base functors \( f^* \colon \cat V^K \to \cat V^J \) for each
  function \( f \colon J \to K \). See also \cite[Corollary 4.9]{Her99}, and
  \cite[Sections 6.2, 6.3]{BJ01}.
\end{proof}

Now, we are ready to prove our main result, Theorem \ref{thm:main.result}. We
begin by considering the following string of functors:
\begin{equation}
  \label{eq:reflector}
  \begin{tikzcd}
    \VCat \ar[r] & \FamVCat \ar[r] & \Cat(\FamV)
  \end{tikzcd}
\end{equation}
By Lemma \ref{lem:main.result}, the functor \( \VCat \to \FamVCat \) reflects
effective descent morphisms. We observe that the same holds for the \(
\FamVCat \to \Cat(\FamV) \): indeed, Lemma \ref{lem:fam.thing} and Remark
\ref{rem:two.examples} allow us to obtain the conclusion of \cite[Theorem
9.11]{Luc18a}. 

Hence, the composite \eqref{eq:reflector} reflects effective descent
morphisms, that is, if \( \tilde F \) is the value of a \( \cat V \)-functor
\( F \colon \cat C \to \cat D \) via \eqref{eq:reflector}, then \( F \) is an
effective descent \( \cat V \)-functor whenever \( \tilde F \) is an effective
descent morphism in \( \Cat(\FamV) \). 

By Le Creurer's \cite[Corollary 3.3.1]{Cre99} and \cite[Lemma A.3]{PL23}, the
problem of confirming \(\tilde F\) is an effective descent morphism can be
reduced to showing that the actions of \( \tilde F \) on \(n\)-tuples of
composable morphisms satisfy suitable ``surjectivity'' conditions for \( n =
1,\, 2,\, 3 \). To be precise, we denote by \( \tilde F^n \) the morphism in
\( \FamV \) given by the action of \( \tilde F \) on the objects of
\(n\)-tuples of composable morphisms. For \(n = 1,\, 2,\, 3 \), we have
\begin{enumerate}[label=(\roman*)]
  \item
    \label{enum:f1}
    \( \tilde F^1 \colon (\cat C(x_0,x_1))_{x_i \in \ob \cat C}
                  \to (\cat D(y_0,y_1))_{y_i \in \ob \cat D} \),
  \item
    \label{enum:f2}
    \( \tilde F^2 \colon (\cat C(x_0,x_1,x_2))_{x_i \in \ob \cat C}
                  \to (\cat D(y_0,y_1,y_2))_{y_i \in \ob \cat D} \),
  \item
    \label{enum:f3}
    \( \tilde F^3 \colon (\cat C(x_0,x_1,x_2,x_3))_{x_i \in \ob \cat C}
                  \to (\cat D(y_0,y_1,y_2,y_3))_{y_i \in \ob \cat D} \).
\end{enumerate}
and these are respectively given by the formal coproducts of the following
morphisms
\begin{align*}
  F_{x_0,x_1} &\colon \cat C(x_0,x_1) \to \cat D(Fx_0,Fx_1) \\
  F_{x_0,x_1,x_2} &\colon \cat C(x_0,x_1,x_2) \to \cat D(Fx_0,Fx_1,Fx_2) \\
  F_{x_0,x_1,x_2,x_3} &\colon \cat C(x_0,x_1,x_2,x_3) 
                          \to \cat D(Fx_0,Fx_1,Fx_2,Fx_3)
\end{align*}
indexed by the underlying functions \( F^n \colon (\ob \cat C)^n \to (\ob \cat
D)^n \), with \( x_i \in \ob \cat C \). Therefore, if
\begin{itemize}[label=--]
  \item
    \( \tilde F^1 \) is an effective descent morphism,
  \item
    \( \tilde F^2 \) is a descent morphism,
  \item
    \( \tilde F^3 \) is an almost descent morphism,
\end{itemize}
then \( \tilde F \) is an effective descent morphism in \( \Cat(\FamV) \). 

If \( \cat E \) is the class of effective descent morphisms, descent morphisms
or almost descent morphisms in \( \FamV \), then \( \cat E \) is closed under
coproducts. 

Therefore, if we are given a morphism \( (f, \phi) \colon (X_j)_{j \in J} \to
(Y_k)_{k \in K} \), we can prove that \( (f,\phi) \in \cat E \) if its
restriction to the fibre at \(k \in K\)
\begin{equation*}
  \phi|_k = (!,\phi) \colon (X_j)_{j \in f^*k} \to Y_k 
\end{equation*} 
is in \( \cat E \) for all \(k \in K\), since \( (f,\phi) \) is the coproduct
of \( \phi|_k \) indexed by \( k \in K \) in the arrow category of \( \FamV
\).

\begin{theorem}
  \label{thm:main.result}
  Let \( F \colon \cat C \to \cat D \) be a \( \cat V \)-functor. If
  \begin{enumerate}[label=\normalfont{(\Roman*)}]
    \item
      \label{enum:ef.desc.sings}
      \( (F_{x_0,x_1})_{x_i \in F^*y_i} \colon (\cat C(x_0,x_1))_{x_i \in
      F^*y_i} \to \cat D(y_0,y_1) \) is an effective descent morphism in \(
      \FamV \) for all pairs \( y_0, y_1 \) of objects in \( \cat D \),
    \item
      \label{enum:desc.pairs}
      \( (F_{x_0,x_1,x_2})_{x_i \in F^*y_i} \colon (\cat C(x_0,x_1,x_2))_{x_i
      \in F^*y_i} \to \cat D(y_0,y_1,y_2) \) is a descent morphism in \( \FamV
      \) for all triples \( y_0,y_1, y_2\) of objects in \( \cat D \),
    \item
      \label{enum:al.desc.trips}
      \(  (F_{x_0,x_1,x_2,x_3})_{x_i \in F^*y_i} \colon (\cat
      C(x_0,x_1,x_2,x_3))_{x_i \in F^*y_i} \to \cat D(y_0,y_1,y_2,y_3) \) is
      an almost descent morphism in \( \FamV \) for all quadruples \( y_0,
      y_1, y_2, y_3 \) of objects in \( \cat D \),
  \end{enumerate}
  then \(F\) is an effective descent morphism in \( \VCat \).
\end{theorem}

\begin{proof}
  By taking the coproducts of the morphisms given in \ref{enum:ef.desc.sings},
  \ref{enum:desc.pairs} and \ref{enum:al.desc.trips}, we conclude that
  \ref{enum:f1}, \ref{enum:f2} and \ref{enum:f3} are respectively an effective
  descent morphism, a descent morphism, and an almost descent morphism. Thus,
  it follows that \( \tilde F \) is an effective descent morphism in \(
  \Cat(\FamV) \). 

  By reflecting along \eqref{eq:reflector}, we conclude that \( F \) is an
  effective descent morphism in \( \VCat \).
\end{proof}

  \section{Familial descent morphisms}
    \label{sect:examples}
    Theorem \ref{thm:main.result} raises the question of understanding
(stable) regular epimorphisms and effective descent morphisms in \( \FamV \)
for a category \( \cat V \) with finite limits, with the goal of providing
more tractable methods to verify conditions \ref{enum:ef.desc.sings},
\ref{enum:desc.pairs}, \ref{enum:al.desc.trips}.

The key ideas for many of the applications are given in the next couple of
lemmas. We begin by noting that the kernel pair of a morphism \( \phi \colon
(X_i)_{i \in I} \to Y \) in \( \FamV \) is calculated by considering the
pullback
\begin{equation}
  \label{eq:fam.kernel.pair}
  \begin{tikzcd}
    X_i \times_Y X_j \ar[r,"\pi^0_{i,j}"] \ar[d,"\pi^1_{i,j}",swap]
      & X_j \ar[d,"\phi_j"] \\
    X_i \ar[r,"\phi_i",swap] & Y
  \end{tikzcd}
\end{equation}
for each \( i,j \in I \). Then, the kernel pair of \( \phi \), denoted by \(
\Ker \phi \), is given by 
\begin{equation*}
  \begin{tikzcd}
    (X_i \times_Y X_j)_{i,j \in I\times I} 
      \ar[r,shift right,"{(p_0,\pi^0)}",swap]
      \ar[r,shift left,"{(p_1,\pi^1)}"]
    & (X_i)_{i\in I}
  \end{tikzcd}
\end{equation*}
where \( p_n \colon I \times I \to I \) for \( n=0,1\) is the projection
which forgets the \(n\)th component.

\begin{lemma}
  \label{lem:fam.reg.epis}
  Let \( \phi \colon (X_i)_{i \in I} \to Y \) be a morphism in \( \Fam(\cat V)
  \). We consider the diagram \( D_\phi \colon \cat J_I \to \cat V \) where
  \begin{itemize}[label=--]
    \item
      \( \ob \cat J_I = (I\times I) \, + \, I \),
    \item
      for each pair \( i,j \in I \), we have two arrows \( (i,j) \to i \) and
      \( (i,j) \to j \),
    \item
      the values of \(D_{\phi} \) at \( (i,j) \to i \) and \( (i,j) \to j \)
      are defined to be \( \pi^1_{i,j} \) and \( \pi^0_{i,j} \), respectively.
  \end{itemize}
  \( \phi \) is a (stable) regular epimorphism if and only if \( D_{\phi} \)
  has a (stable) colimit and \( \colim D_{\phi} \iso Y \).
\end{lemma}

\begin{proof}
  We begin by recalling that a morphism in a category with finite limits is a
  regular epimorphism if and only if it is the coequalizer of its kernel pair.

  The fibration \( \FamV \to \Set \) is a left adjoint functor, hence
  preserves colimits. In particular, if \( \phi \colon (X_i)_{i \in I} \to Y
  \) is a regular epimorphism, then 
  \begin{equation}
    \label{eq:coeq}
    \begin{tikzcd}
      I \times I \ar[r,shift right,"p_0",swap] 
                 \ar[r,shift left,"p_1"] & I \ar[r] & 1
    \end{tikzcd}
  \end{equation}
  must be a coequalizer, and this is the case only when \(I\) is non-empty.

  We note that we have a natural isomorphism
  \begin{equation*}
    \Nat(\ker \phi, \Delta_{(Z_k)_{k \in K}}) \iso 
    \sum_{k \in K} \Nat(D_\phi, \Delta_{Z_k}),
  \end{equation*}
  which is fibered over \(K\): an element from either set is completely
  determined by an element \(k \in K \) and a morphism \( \omega \colon
  (X_i)_{i\in I} \to Z_k \) in \( \FamV \) satisfying \( \omega_i \circ
  \pi^1_{i,j} = \omega_j \circ \pi^0_{i,j} \) for all \(i,j \in I\). Given
  such an element, any morphism \( (q,\psi) \colon (Z_k)_{k\in K} \to (W_l)_{l
  \in L} \) provides an element \( qk \in L \) and a morphism \( \psi_k \circ
  \omega \colon (X_i)_{i \in I} \to W_{qk} \) satisfying \( \psi_k \circ
  \omega_i \circ \pi^1_{i,j} = \psi_k \circ \omega_j \circ \pi^0_{i,j} \) for
  all \(i,j\).

  Thus, if \( \ker \phi \) has a colimit, its underlying set is necessarily a
  singleton by \eqref{eq:coeq}, so we denote it as an object \( Q \) of \(
  \cat V \). We have
  \begin{equation*}
    \sum_{k\in K} \Nat(D_\phi, \Delta_{Z_k})
      \iso \FamV(Q,(Z_k)_{k\in K}) 
      \iso \sum_{k \in K} \cat V(Q,Z_k),
  \end{equation*}
  and since this isomorphism is fibered over \(K\), we conclude \(Q\) is a
  colimit of \( D_\phi \). 

  Conversely, if \( Q \) is a colimit of \( D_\phi \), then we have
  \begin{equation*}
    \Nat(D_\phi,\Delta_{(Z_k)_{k \in K}}) 
      \iso \sum_{k \in K} \cat V(Q,Z_k) 
      \iso \FamV(Q,(Z_k)_k)
  \end{equation*}
  which confirms \(Q\) is a colimit of \( \ker \phi \).

  To verify stability, we begin by assuming \( \phi \) to be a regular
  epimorphism. Given a morphism \( \omega \colon Z \to Y \), the colimits of
  \( \ker \omega^*(\phi) \) and \( D_{\omega^*(\phi)} \) are isomorphic
  whenever either exist, so the stability of one colimit is the equivalent to
  the other. Taking coproducts in \( \FamV \), we confirm the same holds for
  any morphism \( (Z_k)_{k \in K} \to Y \).
\end{proof}

Understanding effective descent morphisms in \( \FamV \) is a more difficult
problem, as is to be expected. However, we can reduce the study of the
category of descent data of a morphism \( \phi \colon (X_i)_{i\in I} \to Y \)
to the full subcategory of \textit{connected} descent data, an idea made
precise by the following result.

\begin{lemma}
  \label{lem:desc.fam.conn.desc}
  Let \( \phi \colon (X_i)_{i \in I} \to Y \) be a morphism in \( \FamV \),
  with \(I\) non-empty. We have an equivalence \( \Desc(\phi) \eqv
  \Fam(\Desc_\conn(\phi)) \), where \( \Desc_\conn(\phi) \) is the full
  subcategory of connected objects of \( \Desc(\phi) \).
\end{lemma}

\begin{proof}
  Given descent data \( (f,\gamma) \), \( (h,\xi) \) as in the following
  diagram
  \begin{equation}
    \label{eq:given.desc.data}
    \begin{tikzcd}
      (W_k \times_Y X_j)_{k,j \in K \times I}
        \ar[d]
        \ar[r,"{(h,\xi)}",shift left]
        \ar[r,"{(p_0,\pi_0)}",shift right,swap]
      & (W_k)_{k \in K} 
        \ar[d,"{(f,\gamma)}"] \\
      (X_i \times_Y X_j)_{i,j \in I\times I}
        \ar[r,shift left,"({p_1,\pi_1})"]
        \ar[r,shift right,"({p_0,\pi_0})",swap]
      & (X_i)_{i \in I} \ar[r,"\phi",swap]
      & Y
    \end{tikzcd}
  \end{equation}
  we obtain descent data \( (f,h) \) for the unique morphism \( I \to 1 \).
  Since \(I\) is non-empty, this morphism is effective for descent, so that \(
  K \iso J \times I \) for a set \(J\), and we may take \(f = p_0 \colon
  J\times I \to I \) and \( h = p_2 \colon J \times I \times I \to J \times I
  \) to be projections (recall \(p_n\) forgets the \(n\)th component).

  Thus, taking the pullback of this descent data along \( ((j,-),\id) \colon
  (W_{j,i})_{i \in I} \to (W_{j,i})_{j,i\in J \times I} \), we obtain the
  following descent data for \( \phi \):
  \begin{equation}
    \begin{tikzcd}[column sep=large]
      (W_{j,i} \times_Y X_k)_{i,k \in I \times I}
        \ar[d]
        \ar[r,"{(p_1,\xi_{j,-,-})}",shift left]
        \ar[r,"{(p_0,\pi_0)}",shift right,swap]
      & (W_{j,i})_{i \in I} 
        \ar[d,"{(\id,\gamma_{j,-})}"] \\
      (X_i \times_Y X_k)_{i,k \in I\times I}
        \ar[r,shift left,"({p_1,\pi_1})"]
        \ar[r,shift right,"({p_0,\pi_0})",swap]
      & (X_i)_{i \in I} \ar[r,"\phi",swap]
      & Y
    \end{tikzcd}
  \end{equation}
  for each \(j \in J \).

  Now, we claim that the descent data of the form
  \begin{equation}
    \label{eq:conn.desc.data}
    \begin{tikzcd}
      (V_i \times_Y X_j)_{i,j \in I \times I}
        \ar[d]
        \ar[r,"{(p_1,\zeta)}",shift left]
        \ar[r,"{(p_0,\pi_0)}",shift right,swap]
      & (V_i)_{i \in I} 
        \ar[d,"{(f,\gamma)}"] \\
      (X_i \times_Y X_j)_{i,j \in I\times I}
        \ar[r,shift left,"({p_1,\pi_1})"]
        \ar[r,shift right,"({p_0,\pi_0})",swap]
      & (X_i)_{i \in I} \ar[r,"\phi",swap]
      & Y
    \end{tikzcd}
  \end{equation}
  is connected in \( \Desc(\phi) \). More concretely, we wish to prove that
  any morphism of descent data \( (q,\chi) \colon (V_i)_{i\in I} \to
  (W_{j,i})_{j,i \in J \times I} \) factors through \( ((j,-),\id) \) for some
  \(j \in J \). This gives a morphism of descent data \( q \colon (\id,p_1)
  \to (p_0,p_2) \) for the unique morphism \( I \to 1 \). We note that \(q\)
  is uniquely determined by a function \( j \colon 1 \to J \), whose value
  provides the desired factorization.

  Having verified that all descent data is a coproduct of connected descent
  data, the result follows.
\end{proof}

\begin{theorem}
  \label{cor:fam.eff.desc}
  Let \( \phi \colon (X_i)_{i \in I} \to Y \) be a morphism in \( \FamV \).
  Then the following are equivalent:
  \begin{enumerate}[label=(\alph*)]
    \item
      \label{enum:fam.eff.desc}
      \( \phi \) is effective for descent.
    \item
      \label{enum:slice.desc}
      We have an equivalence \( \cat V/Y \eqv \Desc_\conn(\phi) \).
  \end{enumerate}
\end{theorem}

\begin{proof}
  We note that the full subcategory of connected objects of \( \FamV / Y \) is
  precisely \( \cat V / Y \), and any object of \( \FamV / Y \) is a coproduct
  of such connected objects. Thus, \( \FamV / Y \eqv \Fam(\cat V/Y) \), and 
  we may conclude \ref{enum:slice.desc} \( \implies \)
  \ref{enum:fam.eff.desc}, since
  \begin{equation*}
    \FamV / Y \eqv \Fam(\cat V/Y) 
              \eqv \Fam(\Desc_\conn(\phi))
              \eqv \Desc(\phi),
  \end{equation*}
  by Lemma \ref{lem:desc.fam.conn.desc}.

  The converse relies on the fact that the comparsion \( \mathcal K^\phi
  \colon \FamV/Y \to \Desc(\phi) \) is of the form \( \Fam(\mathcal
  K^\phi_\conn) \) for a functor \( \mathcal K^\phi_\conn \colon \cat V/Y \to
  \Desc_\conn(\phi) \). Since \( \Fam \) reflects equivalences (because the
  2-natural embedding \( \cat C \to \Fam(\cat C) \) is 2-cartesian), we
  conclude \ref{enum:fam.eff.desc} \( \implies \) \ref{enum:slice.desc}.
\end{proof}

\subsection*{Frames:}

Effective descent morphisms in \( \VCat \) were studied in \cite[Section
5]{CH04}, for Heyting lattices \( \cat V \). As an illustration of our tools,
we provide a second proof that *-quotient morphisms in \( \VCat \) (that is,
surjective-on-objects \( \cat V \)-functors that satisfy condition
\eqref{eq:heyt.reg.epi} below for all \(y_0,y_1,y_2\)) are effective for
descent when \( \cat V \) is a (co)complete Heyting lattice.

Let \( \cat V \) be a thin category (ordered set). A morphism \( (X_i)_{i \in
I} \to Y \) in \( \Fam(\cat V) \) is simply the assertion ``for all \( i \in I
\), \( X_i \leq Y \)''. Thus, we simply write \( (X_i)_{i \in I} \leq Y \) in
this context.

\begin{lemma}
  \label{lem:thin.fam.epis}
  Let \( (X_i)_{i \in I} \leq Y \) be a morphism in \( \Fam(\cat V) \). 
  \begin{itemize}[label=--]
    \item
      It is an epimorphism if and only if \(I\) is non-empty.
    \item
      If it is an epimorphism, it is also stable.
    \item
      It is a regular epimorphism if and only if \( \Join_{i \in I} X_i
      \iso Y \).
    \item
      If it is a regular epimorphism, it is stable if and only if the above
      join is \textit{distributive}, that is,
      \begin{equation}
        \label{eq:dist.cond}
        Z \meet \Join_{i \in I} X_i \iso \Join_{i\in I} Z \meet X_i
      \end{equation}
      for all \( Z \leq Y \).
  \end{itemize}
\end{lemma}

\begin{proof}
  We note that \( (X_i)_{i \in I} \leq Y \) is an epimorphism if and only if
  the underlying function \( I \to 1 \) is surjective, and this is the case
  exactly when \(I\) be non-empty. 

  So, if \(I\) is non-empty, we confirm \( (X_i)_{i \in I} \leq Y \) is a
  stable epimorphism: given \( Z \leq Y \) we can produce an epimorphism \(
  (Z\meet X_i)_{i \in I} \leq Z \), since \( \cat V \) has meets. By taking
  coproducts, the same holds for all \( (Z_j)_{j \in J} \leq Y \).

  We immediately deduce from Lemma \ref{lem:fam.reg.epis}, that \( (X_i)_{i
  \in I} \leq Y \) is a regular epimorphism if and only if \( \Join_{i \in I}
  X_i \iso Y \), and stability under pullbacks is exactly the condition
  \eqref{eq:dist.cond}, so there's nothing to verify.
\end{proof}

We say a thin category \( \cat V \) is a \textit{Heyting semilattice} (also
known as an \textit{implicative semilattice} \cite{Nem65} or a \textit{Brouwerian
semilattice} \cite{Köh81}) if it has finite limits (has meets and is bounded)
and is cartesian closed (has implication). In particular, this means that \( a
\meet - \) is a left adjoint functor for each \( a \in \cat V \), which must
preserve colimits (joins). As a corollary, we conclude that:

\begin{corollary}
  \label{cor:hsl.reg.epi.stab}
  If \( \cat V \) is a Heyting semi-lattice, regular epimorphisms in \(
  \FamV \) are stable.
\end{corollary}
\begin{proof}
  Condition \eqref{eq:dist.cond} is automatically satisfied, by the previous
  remark.
\end{proof}

\begin{corollary}
  \label{cor:chey.reg.epi.ef.desc}
  If \( \cat V \) is a (co)complete Heyting (semi-)lattice, then regular
  epimorphisms in \( \FamV \) are effective for descent.
\end{corollary}
\begin{proof}
  Let \( (X_i)_{i\in I} \leq Y \) be a regular epimorphism. Given connected
  descent data \( \id \colon (W_i)_{i \in I} \to (X_i)_{i \in I} \), \( \pi_2
  \colon (W_i \meet X_j)_{i,j \in I\times I} \to (W_i)_{i \in I} \), we may
  define \( Z = \Join_{i \in I} W_i \). 

  We note that it is enough to prove that \( X_i \meet Z \leq W_i \) for all
  \( i \in I \). Indeed, by distributivity, we have
  \begin{equation*}
    X_i \land Z \iso \Join_{j\in I} X_i \land W_j \leq W_i. 
  \end{equation*}
  Now, Theorem \ref{cor:fam.eff.desc} and \cite[Corollary 2.3]{PL23} complete
  our proof.
\end{proof}

With this, we obtain one direction of \cite[Theorem 5.4]{CH04}:

\begin{theorem}
  \label{thm:chey.cat.ef.desc}
  Let \( \cat V \) be a (co)complete Heyting (semi-)lattice, and let \(F
  \colon \cat C \to \cat D \) be a \( \cat V \)-functor. If \(F\) is
  surjective on objects and we have an isomorphism
  \begin{equation}
    \label{eq:heyt.reg.epi}
    \Join_{x_i \in F^*y_i} \cat C(x_0,x_1,x_2) 
      \iso \cat D(y_0,y_1,y_2) 
  \end{equation}
  for all \(y_0,y_1,y_2 \), then \(F\) is effective for descent.
\end{theorem}

\begin{proof}
  Due to Lemma \ref{lem:thin.fam.epis}, we may conclude that
  \begin{itemize}[label=--]
    \item
      condition \ref{enum:al.desc.trips} is satisfied, since \(F\) is
      surjective on objects, and
    \item
      condition \ref{enum:desc.pairs} is given by
      \eqref{eq:heyt.reg.epi}, plus the stability of regular epimorphisms
      provided by Corollary \ref{cor:hsl.reg.epi.stab}.
  \end{itemize}

  Condition \ref{enum:ef.desc.sings} remains to be verified. Taking \( y_1 =
  y_2 \) above, so that \( \cat D(y_1,y_2) \iso 1 \), we have
  \begin{equation*}
    \cat D(y_0,y_1) \iso 
      \Join_{x_i \in F^*y_i} \cat C(x_0,x_1,x_2)
      \leq \Join_{x_i \in F^*y_i} \cat C(x_0,x_1),
  \end{equation*}
  and since we have \( \cat C(x_0,x_1) \leq \cat D(y_0,y_1) \) for all \(x_i
  \in F^*y_i \), we conclude that \( (\cat C(x_0,x_1))_{x_i \in F^*y_i} \leq
  \cat D(y_0,y_1) \) is a regular epimorphism in \( \FamV \), and therefore is
  effective for descent by Corollary \ref{cor:chey.reg.epi.ef.desc}.
\end{proof}

\subsection*{Regular categories:}

The ideas behind the previous results generalize to regular categories \( \cat
V \), via their (regular epi, mono)-factorization system, which allow us to
reduce statements about epimorphisms \( \phi \colon (X_i)_{i\in I} \to Y \) in
\( \FamV \) to families of monomorphisms. The following result makes this
precise:

\begin{lemma}
  \label{lem:reg.epi.iff.join.img}
  Suppose \( \cat V \) is a regular category, and let \( \phi \colon (X_i)_{i
  \in I} \to Y \) be a morphism in \( \FamV \). For each \(i \in I\), consider
  the following factorization
  \begin{equation}
    \label{eq:fam.fact}
    \begin{tikzcd}
      X_i \ar[r,"\pi_i"] & M_i \ar[r,"\iota_i"] & Y
    \end{tikzcd}
  \end{equation}
  where \( \pi_i \) is a regular epimorphism and \( \iota_i \) is a
  monomorphism for all \(i \in I\).
 
  \begin{itemize}[label=--]
    \item
      \( \phi \) is a (stable) epimorphism if and only if \( \iota \) is a
      (stable) epimorphism.

    \item
      \( \phi \) is a (stable) regular epimorphism if and only if \( \Join_{i
      \in I} M_i \iso Y \) (and the join is stable).

    \item
      If \( \pi_i \) is an effective descent morphism in \( \cat V \) for all
      \(i \in I \), then \( \phi \) is an effective descent morphism if and
      only if \( \iota \) is an effective descent morphism.
  \end{itemize}
\end{lemma}

\begin{proof}
  The factorizations \eqref{eq:fam.fact} for each \( i \in I \) give a
  factorization \( \phi = \iota \circ (\id, \pi) \) in \( \FamV \). We note
  that \( (\id,\pi) \) is a coproduct of stable regular epimorphisms, hence \(
  \phi \) is a (stable) (regular) epimorphism if and only if \( \iota \) is a
  (stable) (regular) epimorphism (see \cite[Propositions 1.3, 1.5]{JST04}).

  Moreover, if \( \pi_i \) is effective for descent for all \(i \in I\),
  taking coproducts will guarantee \( (\id,\pi) \) is effective for descent as
  well, meaning that \( \phi \) is effective for descent if and only if \(
  \iota \) is effective for descent (because the basic bifibration respects
  the BED, see \cite[Section 4]{JT97}).

  Under this light, the results are immediate consequences of Lemma
  \ref{lem:fam.reg.epis}. 
\end{proof}

To prove our main result for this section, we let \( F \colon \cat C \to \cat
D \) be a \( \cat V \)-functor, and we consider the following (regular epi,
mono)-factorizations of the underlying morphisms of the families
\ref{enum:ef.desc.sings}, \ref{enum:desc.pairs}, \ref{enum:al.desc.trips}, as
in \eqref{eq:fam.fact}:
\begin{equation}
  \label{eq:sings.fact}
  \begin{tikzcd}[column sep=large]
    \cat C(x_0, x_1) \ar[r,"P_{x_0,x_1}"]
      & M_{x_0,x_1} \ar[r,"I_{x_0,x_1}"]
      & \cat D(y_0, y_1)
    \end{tikzcd}
\end{equation}
\begin{equation}
  \label{eq:pairs.fact}
  \begin{tikzcd}[column sep=large]
    \cat C(x_0, x_1, x_2) \ar[r]
      & M_{x_0,x_1,x_2} \ar[r,"I_{x_0,x_1,x_2}"]
      & \cat D(y_0, y_1, y_2) 
  \end{tikzcd}
\end{equation}
\begin{equation}
  \label{eq:trips.fact}
  \begin{tikzcd}[column sep=large]
    \cat C(x_0,x_1,x_2, x_3) \ar[r]
      & M_{x_0,x_1,x_2,x_3} \ar[r,"I_{x_0,x_1,x_2,x_3}"] 
      & \cat D(y_0, y_1, y_2, y_3) 
  \end{tikzcd}
\end{equation}
respectively, for objects \( y_i \) in \( \cat D \), and \( x_i \in F^*y_i \),
for \( n = 0,\, 1,\, 2,\, 3 \).

\begin{theorem}
  \label{thm:reg.cat.desc}
  Let \( \cat V \) be a regular category, and let \(F \colon \cat C \to \cat D
  \) be a \( \cat V \)-functor. If
  \begin{enumerate}[label=(\roman*)]
    \item
      \label{enum:ef.desc.img}
      \( P_{x_0,x_1} \) is an effective descent morphism in \( \cat V \) for
      each pair of objects \(x_0, x_1\),
    \item
      \label{enum:ef.desc.monos}
      we have an equivalence \( \cat V / \cat D(y_0,y_1) \eqv \Desc_\conn(I)
      \) for all \( y_0, y_1 \),
    \item
      \label{enum:desc.join.pairs}
      the join \( \Join_{x_i \in F^*y_i} M_{x_0,x_1,x_2} \) exists, is stable
      and is isomorphic to \( \cat D(y_0,y_1, y_2) \) for all \( y_0, y_1, y_2
      \),
    \item
      \label{enum:al.desc}
      \( I \colon (M_{x_0,x_1,x_2,x_3})_{x_i\in F^*y_i} \to \cat
      D(y_0,y_1,y_2,y_3) \) is an almost descent morphism in \( \FamV \) for
      all \( y_0, y_1, y_2, y_3 \),
 \end{enumerate}
  then \(F\) is effective for descent.
\end{theorem}

\begin{proof}
  The goal is to verify that properties \ref{enum:ef.desc.sings},
  \ref{enum:desc.pairs} and \ref{enum:al.desc.trips} are satisfied, so we can
  apply Theorem~\ref{thm:main.result}. 

  By Theorem \ref{cor:fam.eff.desc}, condition \ref{enum:ef.desc.monos} holds
  if and only if \( I \colon (M_{x_0,x_1})_{x_i \in F^*y_i} \to \cat D(y_0,
  y_1) \) is an effective descent morphism for all objects \( y_0, y_1 \) in
  \( \cat D \).

  Now, by applying Lemma~\ref{lem:reg.epi.iff.join.img}, we have that
  \begin{itemize}[label=--]
    \item
      \ref{enum:ef.desc.sings} follows as a consequence of
      \ref{enum:ef.desc.img} and \ref{enum:ef.desc.monos},

    \item
      \ref{enum:desc.pairs} is a consequence of \ref{enum:desc.join.pairs},

    \item
      \ref{enum:al.desc.trips} is a consequence of \ref{enum:al.desc},
  \end{itemize}
  so the result follows.
\end{proof}

In case \( \cat V \) satisfies further properties, we can simplify the above
list:
\begin{itemize}[label=--]
  \item
    If \( \cat V \) is infinitary coherent (has stable, arbitrary unions of
    subobjects), then the join in \ref{enum:desc.join.pairs} exists and is
    stable; one only needs to verify if the isomorphism exists.
  \item
    If \( \cat V \) is exact, or locally cartesian closed, then
    \ref{enum:ef.desc.img} is redundant, since regular epimorphisms are
    effective for descent.
\end{itemize}

  \section{Enrichment in cartesian monoidal categories}
    \label{sect:examples.pt2}
    Theorem \ref{thm:main.result} generalizes Lucatelli's result \cite{Luc18a} about
effective descent \(\cat V\)-functors, by not requiring \( \cat V \) to be
extensive, nor that the induced functor \( - \pt 1 \colon \Set \to \cat V \)
is fully faithful (if it even exists). Thus, we consider examples of enriching
categories \( \cat V \) not satisfying one of those aforementioned properties
(excluding examples such as \( \cat V = \Set, \Top, \Cat \)), dedicating this
section to the study of such categories \( \VCat \). 

\subsection*{Thin categories:}

Thin categories \( \cat V \) with cartesian monoidal structures are
(essentially) bounded meet-semilattices, which we have previously discussed
in Section~\ref{sect:examples}, as an illustrative example. We only briefly
repeat here that the result for (co)complete Heyting lattices \( \cat V \)
admits a particularly nice description (Theorem \ref{thm:chey.cat.ef.desc}),
which was already provided in \cite{CH04} using other techniques.

\subsection*{Colax-pointed categories:}

We consider the colax comma category \( 1 // \Cat \). To be explicit, this has
\begin{itemize}[label=--]
  \item
    objects: pairs \( (\cat C,c) \) where \( \cat C \) is a category and
    \(c \in \ob \cat C\).
  \item
    morphisms \( (\cat C,c) \to (\cat D,d) \): pairs \( (F,f) \) where
    \(F\) is a functor and \( f \colon Fc \to d \) is a morphism in \( \cat D
    \).
  \item
    identity on \( (\cat C,c) \): the pair \( (\id,\id) \).
  \item
    composite of \( (F,f) \colon (\cat C,c) \to (\cat D,d) \) with \(
    (G,g) \colon (\cat D,d) \to (\cat E,e) \): the pair \( (G\circ F, g \circ
    Gf) \).
\end{itemize}

This is the category of strict algebras and colax morphisms for the 2-monad \(
1+-\) on \( \Cat \) (the dual and codual notion is present in \cite{Gra74,
STWW80, CL20}). Hence, by \cite[Corollary 4.9]{Lac05}, \( 1 // \Cat \to \Cat \)
creates products, hence \( 1 // \Cat \) admits a cartesian monoidal structure.

However, \( 1 // \Cat \) is not an extensive category, since it doesn't have
an initial object. It doesn't even have coproducts for any pair of objects:
let \( (\cat C_1,c_1) \) and \( (\cat C_2,c_2) \) be pointed categories, and
we assume this pair has a coproduct \( (\tilde{\cat C},\tilde c) \) in \( 1 //
\Cat \), with coprojections \( (I_i,\iota_i) \colon (\cat C_i,c_i) \to
(\tilde{\cat C},\tilde c) \) for \(i=1, 2\).

Let \( F_i \colon \cat C_i \to \cat D \) be functors, and we suppose we have
morphisms \( f_i \colon Fc_i \to d \) for \( i =1,2 \). These define morphisms
\( (F_i,f_i) \colon (\cat C_i,c_i) \to (\cat D,d) \) for \( i = 1,2 \), so the
universal property guarantees there exists a unique morphism \( (G,g) \colon
(\tilde{\cat C},\tilde c) \to (\cat D,d) \) satisfying \( GI_i = F_i \) and \(
f_i = g \circ G\iota_i \). 

In fact, we can prove that \( G\tilde c \iso Fc_1 + Fc_2 \): if \(h \colon
G\tilde c \to d \) is such that \( f_i = h \circ G\iota_i \) for \(i=1,2\),
then \( (G,h) \circ (I_i,\iota_i) = (F_i,f_i) \), and by the universal
property, \( (G,g) = (G,h) \), hence \( g=h \). 

But there is no reason for \( \cat D \) to have such a coproduct: consider the
category \( \cat D \) given by the following graph
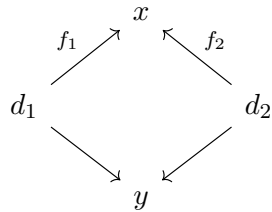
\begin{equation*}
  \begin{tikzcd}
    & x \\
    d_1 \ar[ur,"f_1"] \ar[dr]
    && d_2 \ar[ul,"f_2",swap] \ar[dl] \\
    & y
  \end{tikzcd}
\end{equation*}
and observe that the pair \(d_1,d_2\) does not have a coproduct. Thus, we
obtain the desired contradiction by letting \( F_i \colon \cat C_i \to \cat D
\) be the constant functor to \(d_i\), for \(i=1,2\).

A \( (1 // \Cat) \)-category is a 2-category \( \bicat B \) and
\begin{itemize}[label=--]
  \item
    for each \(x,y\), an object \( \hom(x,y) \in \bicat B(x,y) \),
  \item
    for each \(x\), a morphism \( e_x \colon 1_x \to \hom(x,x) \),
  \item
    for each \(x,y,z\), a morphism \( m_{x,y,z} \colon \hom(y,z) \cdot
    \hom(x,y) \to \hom(x,z) \),
  \item
    the following diagrams commute for all \(w,x,y,z\):
    \begin{equation*}
      \begin{tikzcd}
        \hom(x,y) \cdot 1_x \ar[d,"\id \cdot e_x",swap]
                            \ar[rd,equal] \\
        \hom(x,y) \cdot \hom(x,x) \ar[r,"m_{x,x,y}",swap]
        & \hom(x,y)
      \end{tikzcd}
      \quad
      \begin{tikzcd}
        & 1_y \cdot \hom(x,y) \ar[d,"e_y \cdot \id"]
                              \ar[ld,equal] \\
        \hom(x,y) 
          & \hom(y,y) \cdot \hom(x,y) \ar[l,"m_{x,y,y}"]
      \end{tikzcd}
    \end{equation*}
    \begin{equation*}
      \begin{tikzcd}[column sep=tiny]
        & (p_{y,z} \cdot p_{x,y}) \cdot p_{w,x}
          \ar[rr,equal]
          \ar[ld,"m_{x,y,z} \cdot \id",swap] 
        && p_{y,z} \cdot (p_{x,y} \cdot p_{w,x}) 
          \ar[rd,"\id \cdot m_{w,x,y}"] \\
        p_{x,z} \cdot p_{w,x} \ar[rrd,"m_{w,x,z}",swap]
        &&&& p_{y,z} \cdot p_{w,y} \ar[lld,"m_{w,y,z}"] \\
        && p_{w,z}
      \end{tikzcd}
    \end{equation*}
\end{itemize}
which is a \( \bicat B \)-enriched category on the same set of objects.

A \( (1 // \Cat) \)-functor \( (F,\Phi) \colon (\bicat B, \hom_\bicat B) \to
(\bicat C, \hom_\bicat C) \) consists of a 2-functor \( F \colon \bicat B \to
\bicat C \), and a \( \bicat C \)-functor \( \Phi \colon F_!\bicat B \to
\bicat C \), with the same underlying function on objects. 

\subsection*{Categories with zero object:}

Let \( \cat V \) be a category with a zero object, which we denote by \(1\). 
Such categories are usually not extensive, for if the zero object were strict,
we would have \( \cat V \eqv 1 \).

For a \( \cat V \)-category \( \cat C \), we write \( p_{x,y} \colon 1 \to
\cat C(x,y) \) for the uniquely determined morphism. In particular, this
implies \( \uu x = p_{x,x} \) for all \(x\), and 
\begin{equation*}
  \begin{tikzcd}[column sep=large]
    1 \ar[r,"{(p_{y,x},p_{x,y})}"]
      & \cat C(y,x) \times \cat C(x,y) \ar[r,"\cc{x,y,x}"]
      & \cat C(x,x)
  \end{tikzcd}
\end{equation*}
must also equal \( \uu x \). 

With this, we can confirm that all hom-objects must be isomorphic: the
isomorphism is given by:
\begin{equation*}
  \begin{tikzcd}[column sep=large]
    \cat C(x,y)
      \ar[rr,"p_{y,z} \times \id \times p_{w,x}"]
    && \cat C(y,z) \times \cat C(x,y) \times \cat C(w,x)
      \ar[rr,"\cc{w,y,z} \circ (\id \times \cc{w,x,y})"]
    && \cat C(w,z)
  \end{tikzcd}
\end{equation*}
Thus, we conclude \( \VCat \) has objects the empty \( \cat V \)-category plus
pairs (non-empty set, \(\cat V\)-monoid).

\subsection*{Eckmann-Hilton:}

We suppose \( \cat V \) is the category of unital magmas. By the
Eckmann-Hilton argument, a \( \cat V \)-monoid is precisely a commutative
monoid. Since \( \cat V \) has a zero object, we conclude \( \VCat \)
essentially has objects the empty \( \cat V \)-category plus pairs (non-empty
set, commutative monoid), in which case the effective descent \( \cat V
\)-functors are given by the empty \( \cat V \)-functor on the empty \( \cat V
\)-category, and pairs (surjective function, regular epimorphism of monoids).

\subsection*{Coextensive categories:}

We say a category \( \cat V \) with finite limits 
\begin{itemize}[label=--]
  \item
    has \textit{codisjoint} products if \( \cat V^\op \) has disjoint
    coproducts,
  \item
    has \textit{a strict terminal object} if \( \cat V^\op \) has a strict
    initial object,
  \item
    is \textit{finitely coextensive} if \( \cat V^\op \) is finitely
    extensive.
\end{itemize}

As expected, finitely coextensive categories \( \cat V \) have codisjoint
products and a strict terminal object. This is the case for the categories of
commutative \(R\)-algebras for a ring \(R\), as a class of examples.

We verify that cartesian monoidal categories \( \cat V \) with codisjoint
products and strict terminal object do not provide an interesting enriching
base with the cartesian monoidal structure: we shall confirm that \( \VCat
\eqv \Set \).  Therefore, the effective descent \( \cat V \)-functors are
precisely those that are surjective on objects.

Let \( \cat C \) be a \( \cat V \)-category. For each \(x \in \ob \cat C \),
the unit morphism \( 1 \to \cat C(x,x) \) is an isomorphism, and for each pair
\(x,y \in \ob \cat C \), the composition morphism \( \cat C(x,y) \times \cat
C(y,x) \to 1 \) is uniquely determined. Thus, the associativity condition for
elements in \( \cat C(x,y,x,y) \) translates to saying that the projections on
the first and third component
\begin{equation*}
  \begin{tikzcd}
    \cat C(x,y) \times \cat C(y,x) \times \cat C(x,y) 
      \ar[r,"! \times \id",swap,shift right]
      \ar[r,"\id \times !",shift left]
      & \cat C(x,y)
  \end{tikzcd}
\end{equation*}
are equal. But since products are codisjoint, we must have \( \cat C(x,y) \iso
1 \), for all \(x,y\).

\subsection*{Categories of spaces:}

Since most varieties of algebras \( \cat V \) seem to have an uninteresting
\( \VCat \) for the cartesian monoidal structure, we turn our attention to
categories of spaces. 

We begin by instantiating our results when \( \cat V = \CHaus \) of compact
Hausdorff spaces, which is a pretopos \cite{MR20}, and therefore is a
Barr-exact category, but it is not infinitary extensive. Let \( F \colon \cat
C \to \cat D \) be a \( \CHaus \)-functor between \( \CHaus \)-categories, and
consider the factorizations \eqref{eq:sings.fact}, \eqref{eq:pairs.fact} and
\eqref{eq:trips.fact} of the hom-morphisms of \(F\). By Theorem
\ref{thm:reg.cat.desc}, \(F\) is effective for descent if 
\begin{itemize}[noitemsep,label=--]
  \item
    \( \CHaus / \cat D(y_0,y_1) \eqv \Desc_\conn(I) \) for all objects \( y_0,
    y_1 \) in \( \cat D \),
  \item
    We have a stable join \( \Join_{x_0,x_1,x_2} M_{x_0,x_1,x_2} \eqv \cat
    D(y_0,y_1,y_2) \), for all \( y_0 \), \( y_1 \), \(y_2\),
  \item
    \( I \colon (M_{x_0,x_1,x_2,x_3})_{x_i} \to \cat D(y_0,y_1,y_2,y_3) \) is
    an almost descent morphism in \( \Fam(\CHaus)\) for all \( y_0 \), \( y_1
    \), \(y_2\), \(y_3\).
\end{itemize}

Similarly, since the category \( \cat V = \Stn \) of Stone spaces is a regular
category \cite{MR20}, we can also say something about \( \Stn \)-functors. Let
\( F \colon \cat C \to \cat D \) be a \( \Stn \)-functor between \( \Stn
\)-categories, and we consider the factorizations \eqref{eq:sings.fact},
\eqref{eq:pairs.fact} and \eqref{eq:trips.fact} of the hom-morphisms of \(F\).
Again by Theorem \ref{thm:reg.cat.desc}, \( F \) is effective for descent if 
\begin{itemize}[noitemsep,label=--]
  \item
    \( P_{x_0,x_1} \) is an effective descent morphism in \( \Stn \) for each
    pair \( x_0,x_1 \),
  \item
    \( \Stn / \cat D(y_0,y_1) \eqv \Desc_\conn(I) \) for all objects \( y_0,
    y_1 \) in \( \cat D \),
  \item
    We have a stable join \( \Join_{x_0,x_1,x_2} M_{x_0,x_1,x_2} \eqv \cat
    D(y_0,y_1,y_2) \), for all \( y_0 \), \( y_1 \), \(y_2\),
  \item
    \( I \colon (M_{x_0,x_1,x_2,x_3})_{x_i} \to \cat D(y_0,y_1,y_2,y_3) \) is
    an almost descent morphism in \( \Fam(\Stn)\) for all \( y_0 \), \( y_1
    \), \(y_2\), \( y_3 \).
\end{itemize}

It would be worthwhile to explore whether analogous arguments are fruitful in
the more general setting of (monad, frame)-enriched categories. More
specifically, we may take \( \cat V \) to be the category of ordered compact
Hausdorff spaces \cite{Tho09}, or the category of ultrametric compact
Hausdorff spaces.

Finally, we may consider a topos \( \cat V \), so that the effective descent
morphisms are exactly the epimorphisms. It is not always the case that there
exists a functor \( - \pt 1 \colon \Set \to \cat V \), and when it does, it is
not always fully faithful\footnote{Grothendieck toposes satisfying this
property are said to be \textit{hyperconnected}.}, so we can use Theorem
\ref{thm:main.result} expand our knowledge of effective descent \( \cat V
\)-functors to all toposes \( \cat V \). Let \( F \colon \cat C \to \cat D \)
be a \( \cat V \)-functor between \( \cat V \)-categories. If

\begin{itemize}[noitemsep,label=--]
  \item
    \(  (F_{x_0,x_1})_{x_i \in F^*y_i} \colon (\cat C(x_0,x_1))_{x_i \in
    F^*y_i} \to \cat D(y_0,y_1) \) is an effective descent morphism in \(
    \FamV \) for all pairs \( y_0, y_1 \) of objects in \( \cat D \),
  \item
    \( (F_{x_0,x_1,x_2})_{x_i \in F^*y_i} \colon (\cat C(x_0,x_1,x_2))_{x_i
    \in F^*y_i} \to \cat D(y_0,y_1,y_2) \) is a descent morphism in \( \FamV
    \) for all triples \( y_0,y_1, y_2\) of objects in \( \cat D \),
  \item
    \(  (F_{x_0,x_1,x_2,x_3})_{x_i \in F^*y_i} \colon (\cat
    C(x_0,x_1,x_2,x_3))_{x_i \in F^*y_i} \to \cat D(y_0,y_1,y_2,y_3) \) is
    an almost descent morphism in \( \FamV \) for all quadruples \( y_0,
    y_1, y_2, y_3 \) of objects in \( \cat D \),
\end{itemize}
then \(F\) is an effective descent \( \cat V \)-functor. While these
conditions can be refined in general with Theorem \ref{thm:reg.cat.desc}, if
\( \cat V \) is a Grothendieck topos, then we can take advantage of the fact
that \( \FamV \eqv \cat V \downarrow \Set \) is a Grothendieck topos as well,
in which case \(F\) is an effective descent functor whenever
\begin{equation*}
 (F_{x_0,x_1,x_2,x_3})_{x_i \in F^*y_i} \colon 
    (\cat C(x_0,x_1,x_2,x_3))_{x_i \in F^*y_i} \to \cat D(y_0,y_1,y_2,y_3)
\end{equation*}
is an epimorphism in \( \FamV \) for all \( y_0,y_1,y_2,y_3 \).

  \section{Future work}
    \label{sect:future.work}
    Having established sufficient conditions for effective descent in \( \VCat \)
for cartesian monoidal categories \( \cat V \), an obvious continuation would
be to extend this result to suitable monoidal categories \( \cat V \).  We
describe a strategy which would rely on the present work; we denote \(
\CartCat \) and \( \SymMndCat \) for the 2-categories of cartesian (monoidal)
categories and symmetrical monoidal categories. 

Provided \( \CartCat \) has needed (strict) codescent objects, the left
2-adjoint (biadjoint) (pseudo)functor of the forgetful 2-functor \(
\CartCat \to \SymMndCat \) exists; the existence and an explicit description
of such a left 2-adjoint (biadjoint) would be provided via the biadjoint
triangle theorem \cite[Theorem 4.4]{Luc16} (see also \cite[Theorem
2.3]{Luc18b}):
\begin{equation*}
  \begin{tikzcd}
    \CartCat \ar[rr] \ar[rd] 
      && \SymMndCat \ar[ld] \\
    & \Cat
  \end{tikzcd}
\end{equation*}
where every 2-functor is forgetful. Both functors to \( \Cat \) have left
2-adjoints which are easy to describe.

So, if the existence of the left biadjoint \( F \colon \SymMndCat \to
\CartCat \) is guaranteed, we need to study the following questions:
\begin{itemize}[label=--]
  \item
    What conditions on \( \cat V \) guarantee existence of pullbacks in \(
    F\cat V \)? 
  \item
    Is the unit \( \eta \colon \cat V \to F\cat V \) fully faithful?
\end{itemize}

After obtaining solutions to the above questions, we could then study the functor
\begin{equation*}
  \eta_! \colon \VCat \to F\VCat,
\end{equation*}
which raises the ultimate question: does it reflect effective descent
morphisms? An affirmative answer would provide a string of functors
\begin{equation*}
  \begin{tikzcd}
    \VCat \ar[r] & F\VCat \ar[r]
                 & \Fam(F \cat V)\dash \Cat \ar[r]
                 & \Cat(\Fam(F\cat V))
  \end{tikzcd}
\end{equation*}
that reflect effective descent. Then, since \( F\cat V \) is hypothetically a
cartesian monoidal category with finite limits, we obtain a more general
result via Theorem~\ref{thm:main.result}. Combined with an adequate study of
effective descent morphisms in \( F\cat V \), these results can be applied in
the study of effective descent morphisms in \( \VCat \) for any symmetrical
monoidal category \( \cat V \).

\iffalse
Another line of research would be the extension of our results to \( (T,\cat
V) \)-categories, for \( \cat V \) a cartesian monoidal, not necessarily thin,
category. A strategy similar to the one presented in this work would rely on
knowledge of effective descent morphisms in internal multicategories, a
project that was carried out in \cite{PL23}, followed by obtaining an
embedding \( \TVCat \to \CatTV \), which is a topic for future work.
\fi

  \printbibliography

\end{document}